\documentclass[12pt,a4paper]{article}

\usepackage{amssymb}
\usepackage{amsfonts}
\usepackage{amsmath}
\usepackage{amsthm}
\usepackage[english]{babel}
\usepackage[cp1250]{inputenc}
\usepackage[T1]{fontenc}
\usepackage[pdftex]{graphicx}

\newtheorem{thm}{Theorem}

\newtheorem{lem}[thm]{Lemma}
\newtheorem{rmk}{Remark}

\title{On fractional Bessel equation and the description of corneal topography}
\author{Wojciech Okrasi\'nski, \L ukasz P\l ociniczak}

\begin{document}
\maketitle


\begin{abstract}
In this note we apply a modified fractional Bessel differential equation to the problem of describing corneal topography. We find the solution in terms of the power series. This solution has an interesting behavior at infinity which is a generalization of the classical results for modified Bessel function of order 0. Our model fits the real corneal geometry data with an error of order of a few per cent. 
\end{abstract}

\section{Introduction}
Sight is the most crucial sense that we posses since it enables us to perceive the world very accurately. With the advance of medical technology treating various eye diseases becomes more adequate and successful. This would not be possible without proper mathematical models of biomechanics of eye and its constituents. One of the most important parts of the human eye is the cornea because it is responsible for about two-thirds of refractive power (for biological treatment see for ex. \cite{Cornea}). Mathematical description of corneal topography is very important from the point of view of ophthalmologists because many seeing disorders originate in some distortions in corneal geometry.

There are many types of corneal topography mathematical models. The most common and simple are based on conic sections (see for ex. \cite{Conic, Eccen}). Unfortunately, they are taken without much physical motivation. Some, more complex, models are based on finite element methods or shell theory \cite{Shell, FEM}. Models that are widely used to describe abberations in cornea or lens often use Zernike orthogonal polynomials \cite{Zernike,ZernikeRat}.

In \cite{Bessel} we have proposed yet another model of corneal topography. It was based on physical derivation and membrane equation. The first approximation to that model is
\begin{equation}
	-\Delta h + a h =b, \quad h|_{\partial \Omega}=0,
\label{MainEq}
\end{equation}
where $a$ and $b$ are dimensionless, positive constants and $\Omega$ is the domain on which cornea is situated. In this letter we generalize this model and use fractional derivatives instead of classical ones. As a result we obtain a modified fractional Bessel differential equation which we solve and find interesting asymptotic behavior. The fractional calculus methods are increasingly more popular and very successful in a large number of physical application (see for ex. \cite{Heymans, Ervin, Turner}). A comprehensive introductions to fractional calculus and its applications can be found for expample in \cite{Kilbas, Podlubny}. Up to authors' knowledge the fractional Bessel equation was investigated only in \cite{Rodrigues}. However, the form of that equation was different from the one analyzed by us.

\section{Model and its analysis}
Assume that the corneal surface is axisymmetric, that is $h=h(r)$ where $r\in[0,1]$ and $\Omega$ is a unit circle. We also assume that $h$ can be represented by convergent power series. Rewriting (\ref{MainEq}) in polar coordinates we obtain
\begin{equation}
	-\frac{1}{r}\frac{d}{dr}\left(r h'\right)+a h =b, \quad h'(0)=0, \quad h(1)=0.
\label{polar}
\end{equation}
The condition at $r=0$ grants us a smooth solution at the origin while $h(1)=0$ determines the rim of cornea.
Now, multiply (\ref{polar}) by $r^2$, let $y=-h$ and $x=\sqrt{a}r$ to get
\begin{equation}
	x\frac{d}{dx}\left(x y'\right)-x^2 y=\frac{b}{a}x^2,
\label{polar2}
\end{equation}
which is a nonhomogeneous modified Bessel equation of order $0$. Its solution and application to corneal topography was presented in \cite{Bessel}. Now, we want to generalize this model and use fractional derivative instead of the usual one. We propose
\begin{equation}
	x^{\alpha} D_0^\alpha ( x y')-x^2 y=\frac{b}{a}x^2, \quad y'(0)=0, \quad y(\sqrt{a})=0, \quad 0<\alpha\leq1,
\label{FracBessel}
\end{equation}
as the model of corneal geometry. Operator $D_0^{\alpha}$ is the Riemann-Liouville fractional derivative defined by $D_0^{\alpha}y := \frac{d}{dx} D_0^{-(1-\alpha)}y$, where 
\begin{equation}
	(D_0^{-\mu}y)(x):=\frac{1}{\Gamma(\mu)}\int_0^x {(x-t)^{\mu-1}y(t)dt}
\end{equation}
is the fractional integral operator (see \cite{Kilbas}). It can be shown (see \cite{Podlubny,Li}) that in our setting since $xy'(x)|_{x=0}=0$ and $y$ is analytic the Riemann-Lioville and Caputo fractional derivatives coincide. The choice of appropriate power of $x$ in (\ref{FracBessel}) is neccesary to maintain dimensional consistency and physical meaning. We will call (\ref{FracBessel}) the modified fractional Bessel equation of order $0$ in analogy with the classical case.

Immediately we see that $y_p(x)=-b/a$ is the particular solution of (\ref{FracBessel}). Because the general solution to (\ref{FracBessel}) has the form $y=y_H+y_p$, we only have to find the solution $y_H$ to the homogeneous equation
\begin{equation}
	x^{\alpha} D_0^\alpha ( x y')=x^2 y, \quad y'(0)=0, \quad 0<\alpha\leq1,
\label{FracBesselHom}
\end{equation}
where the condition $y'(0)=0$ comes from the fact that $y_p$ is constant. We seek for a solution to (\ref{FracBesselHom}) in the form of power series, that is
\begin{equation}
	y(x)=\sum_{n=0}^{\infty} {a_n x^n}. 
\end{equation}
Substituting it into (\ref{FracBesselHom}) and noting that $(D_0^\alpha t^n)(x)=\Gamma(n+1)/\Gamma(n-\alpha+1)x^{n-\alpha}$ we obtain recurrence formulas for coefficients $a_n$
\begin{equation}
	a_n=\frac{\Gamma(n-\alpha+1)}{n^2 \Gamma(n)}a_{n-2}, \quad n\geq2.
\label{coeff}
\end{equation}
The equations for $a_0$ and $a_1$ are automatically fulfilled since $0=y'(0)=a_1$ and we can choose $a_0=1$. By the ratio test we see that the series is absolutely convergent. Thus, taking into account (\ref{coeff}) we obtain the solution to homogeneous fractional Bessel equation (\ref{FracBesselHom}) which in accordance with classical theory we will denote by $I_0^\alpha$
\begin{equation}
	y_H(x)=C I_0^\alpha (x)=C\sum_{n=0}^\infty{\left(\prod_{i=1}^n \frac{\Gamma(2i-\alpha+1)}{\Gamma(2i)} \right) \frac{1}{n!^2}\left(\frac{x}{2}\right)^{2n}}, \quad 0<\alpha\leq1,
\label{I0}
\end{equation}
for some constant $C$ and the convention that $\prod_{i=1}^0 =1$. This modified fractional Bessel function is a generalization of the classical Bessel function since $I_0^1=I_0$. Additionally, we can see that $I_0^0=\exp(x^2/2)$. Noting the other boundary condition $y(\sqrt{a})=0$ and returning to the original variables $h$ and $r$ we can write the solution to the boundary value problem (\ref{FracBessel})
\begin{equation}
	h(r)=\frac{b}{a}\left(1-\frac{I_0^\alpha(\sqrt{a}r)}{I_0^\alpha(\sqrt{a})} \right).
\label{Sol}
\end{equation}
This equation describes the shape of human cornea and we will see later that it gives very accurate fit with the real data.

The modified fractional Bessel function $I_0^\alpha$ defined in (\ref{I0}) has very interesting behavior as $x\rightarrow\infty$. As we will see it is a generalization of the classical results from asymptotic theory. First, we prove a technical lemma.

\begin{lem}
Let $F_{\mu}(x):=\int_0^1 {t^{\mu}e^{-x t}dt}$, then $F_{\mu}$ has the following leading order behavior as $x\rightarrow\infty$
\begin{equation}
	F_{\mu}(x) \sim x^{-(\mu+1)}\Gamma(\mu+1)
\end{equation}
\label{Lem}
\end{lem}

\begin{proof}
With a change of variable $\left( s= x t \right)$ we can write the integral defining $F_{\mu}$ as
\begin{equation}
	F_{\mu}(x)=x^{-(\mu+1)} \int_0^x {s^{\mu} e^{-s} ds}=x^{-(\mu+1)} \left(\Gamma(\mu+1)-\int_x^\infty {s^{\mu} e^{-s} ds}\right).
\end{equation}
Now, integrating by parts we write the asymptotic expansion of $F_{\mu}$
\begin{equation}
	F_{\mu}(x) \sim x^{-(\mu+1)}\Gamma(\mu+1) -e^{-x} \sum_{k=0}^\infty {\frac{\Gamma(\mu+1)}{\Gamma(\mu-k+1)} \frac{1}{x^{k+1}}},
\label{FmuAsym}
\end{equation}
as $x\rightarrow\infty$. Since the second term in (\ref{FmuAsym}) is exponentially small the leading order behavior of $F_{\mu}$ is defined by te first term and that concludes the proof. 
\end{proof}

The proof of following theorem is based on the Laplace method for asymptotic integrals. The main point is that we are dealing with integro-differential equation which introduces some complications. 

\begin{thm}
The solution to modified Bessel fractional differential equation (\ref{FracBessel}) has the following asymptotic leading order behavior as $x\rightarrow\infty$
\begin{equation}
	I_0^\alpha (x) \sim x^{-\frac{\alpha\left(2-\alpha\right)}{1+\alpha}} \exp\left(\frac{1+\alpha}{2} x^{\frac{2}{1+\alpha}}\right).
\label{Asym}
\end{equation}
\end{thm}

\begin{rmk}
We see that for $\alpha=0,1$ asymptotic form (\ref{Asym}) reduces to well known formulas for $I_0^0(x)=\exp(x^2/2)$ and $I_0^1(x)=I_0(x)\sim x^{-1/2}\exp\left(x\right)$ (see for ex. \cite{AS}).
\label{Rem}
\end{rmk}

\begin{proof}
By using the identity for composition of fractional integral $D_0^{-\alpha}$ and Riemann-Lioville fractional derivative (\cite{Podlubny}, formulas 2.113 and 2.135-136) and by the fact that $xy'(x)|_{x=0}=0$ we have
\begin{equation}
	(D_0^{-\alpha} D_0^\alpha (xy'(x)))(x)=x y'(x) \quad 0<\alpha<1.
\end{equation}
We transform (\ref{FracBessel}) into an Volterra integro-differential equation
\begin{equation}
	y'(x)=\frac{1}{\Gamma(\alpha)}\frac{1}{x} \int_0^x {(x-t)^{\alpha-1}t^{2-\alpha} y(t)dt}.
\label{IntEq}
\end{equation}
Since we are looking for the behavior of (\ref{IntEq}) for large $x$ we change the variable $s=t/x$ to obtain constant limits of integration
\begin{equation}
	y'(x)=\frac{x}{\Gamma(\alpha)} \int_0^1 {(1-s)^{\alpha-1}s^{2-\alpha} y(s x)ds }.
\end{equation}
Now, guided by Remark \ref{Rem} we seek the approximate solution of (\ref{IntEq}) in the form $y(x)=f(x)\exp\left( x^\lambda/\lambda \right)$, where $f$ and $\lambda$ are to be determined and $f$ has algebraic growth. We have
\begin{equation}
	y'(x)=\frac{x}{\Gamma(\alpha)} \int_0^1 {(1-s)^{\alpha-1}s^{2-\alpha} f(s x)e^{\frac{x^\lambda}{\lambda} s^\lambda} ds }.
\label{IntEq2}
\end{equation}
Notice that for large $x$ the integrand in (\ref{IntEq2}) is dominated by exponential term which has its maximum at $s=1$. To move this maximum to the lower limit we substitute again $t=1-s$ and obtain
\begin{equation}
	y'(x)=\frac{x}{\Gamma(\alpha)} \int_0^1 {(1-t)^{2-\alpha} t^{\alpha-1} f(x(1-t))e^{ \frac{x^\lambda}{\lambda} (1-t)^\lambda} dt }.
\label{IntEq3}
\end{equation} 
As in the Laplace method for asymptotic integrals (see for ex. \cite{Bender}) for $x\rightarrow\infty$ the greatest contribution to the integral (\ref{IntEq}) comes from the neighborhood of the maximum of exponential term, that is when $t$ is close to $0$. Writing $(1-t)^\lambda=1-\lambda t+...$, $f(x(1-t))=f(x)-x f'(x) t+...$ and retrieving only first terms we can make an approximation valid for large x 
\begin{equation}
\begin{split}
	y'(x) & \approx \frac{x}{\Gamma(\alpha)} e^{ \frac{x^\lambda}{\lambda} } \left( f(x) \int_0^1 {(1-t)^{2-\alpha} t^{\alpha-1} e^{-x^\lambda t}dt} \right. \\
			  & \left. + x f'(x) \int_0^1 {(1-t)^{2-\alpha} t^{\alpha} e^{-x^\lambda t}dt} \right).
\end{split}
\label{IntEq4}
\end{equation}
In the first integral in (\ref{IntEq4}) we further approximate $(1-t)^{2-\alpha}\approx 1-(2-\alpha)t$ and in the second $(1-t)^{2-\alpha}\approx1$. This is justified since the whole mass of the integral is focused near $t=0$. Using Lemma \ref{Lem} we can write
\begin{equation}
\begin{split}
	y'(x) & \approx \frac{x}{\Gamma(\alpha)} e^{ \frac{x^\lambda}{\lambda} } \left[ f(x) \left(F_{\alpha-1}(x^\lambda)-(2-\alpha)F_\alpha(x^\lambda)\right) + xf'(x)F_\alpha(x^\lambda) \right] \\
	& \sim e^{ \frac{x^\lambda}{\lambda} } \left[ \left(x^{-\lambda\alpha+1}-\alpha(2-\alpha)x^{-\lambda(\alpha+1)+1}\right)f(x) -\alpha x^{2-\lambda(\alpha+1)}f'(x) \right],
\end{split}
\end{equation}
where we used the formula $\Gamma(\alpha+1)=\alpha\Gamma(\alpha)$. Noticing that $y'(x)=f'(x)\exp(x^\lambda/\lambda)+x^{\lambda-1}\exp(x^\lambda/\lambda)$ we finally obtain a differential equation for $f$
\begin{equation}
	\left(1+\alpha x^{2-\lambda(\alpha+1)}\right)f' = \left(x^{-\lambda\alpha+1}-x^{\lambda-1}-\alpha(2-\alpha)x^{-\lambda(\alpha+1)+1}\right)f.
\end{equation}
We see that the choice of $\lambda=2/(1+\alpha)$ is not only necessary for $f$ to have algebraic growth but also greatly simplifies this equation into 
\begin{equation}
	\frac{f'}{f}=-\frac{\alpha(2-\alpha)}{1+\alpha}\frac{1}{x} ,
\end{equation}
which has the solution
\begin{equation}
	f(x)=C x^{-\frac{\alpha(2-\alpha)}{1+\alpha}}.
\end{equation}
This finishes the proof.
\end{proof}

\section{Application to corneal topography}
In this section we apply previously analyzed mathematical model (\ref{Sol}) to a dataset consisting of $123\times123$ measurement points of real cornea. We fit the solution $h$ finding the least squares unknown parameters $a$, $b$ and $\alpha$. Cornea has its thickness thus we make two fits: one for exterior and one for interior surface. For exterior surface we find $a=0.580404$, $b=1.19734$ and $\alpha=0.421345$ with a mean absolute fitting error $0.014mm$. Similarly, for interior surface we have $a=0.818763$, $b=1.66664$, $\alpha=0.503431$ and the mean absolute fitting error $0.03mm$. We see that both parameters $a$ and $b$ are of order of unity what was expected (see \cite{Bessel}) and the order of derivative in (\ref{FracBessel}) lies between $1$ and $2$. Contour plots of absolute fitting errors are depicted on Figure \ref{Err}.

\begin{figure}[htb]
	\centering
	\includegraphics[scale=0.65]{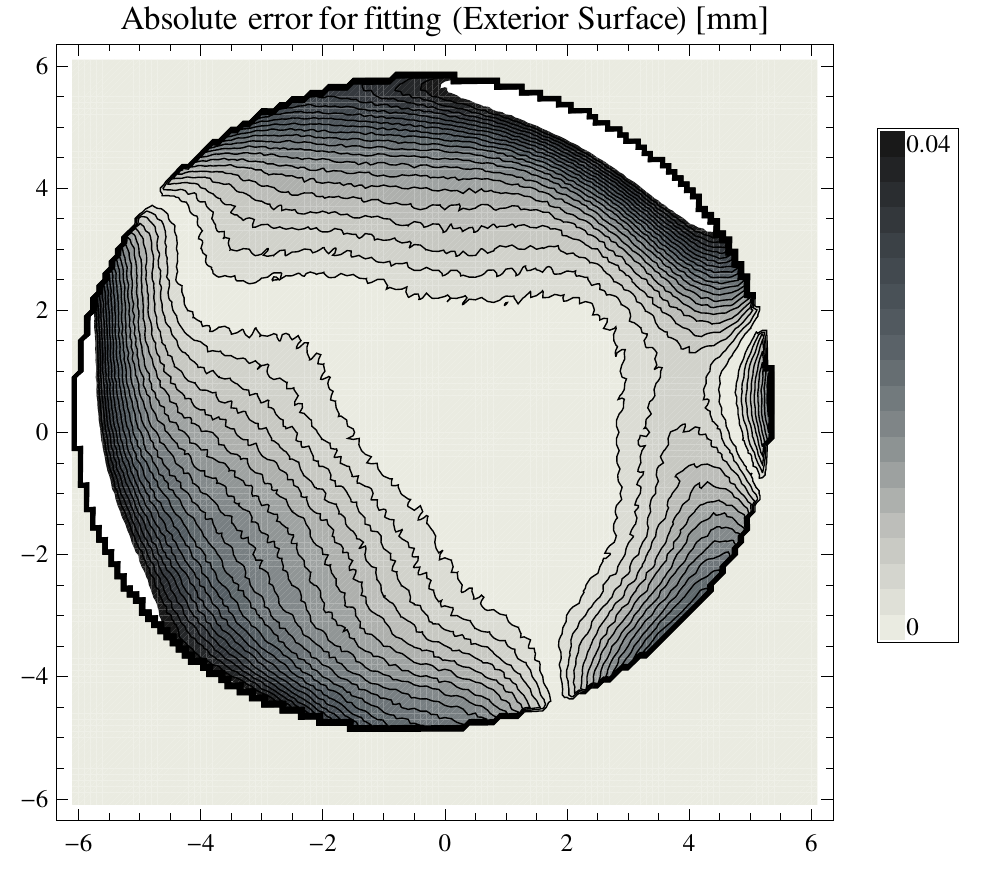}
	\includegraphics[scale=0.65]{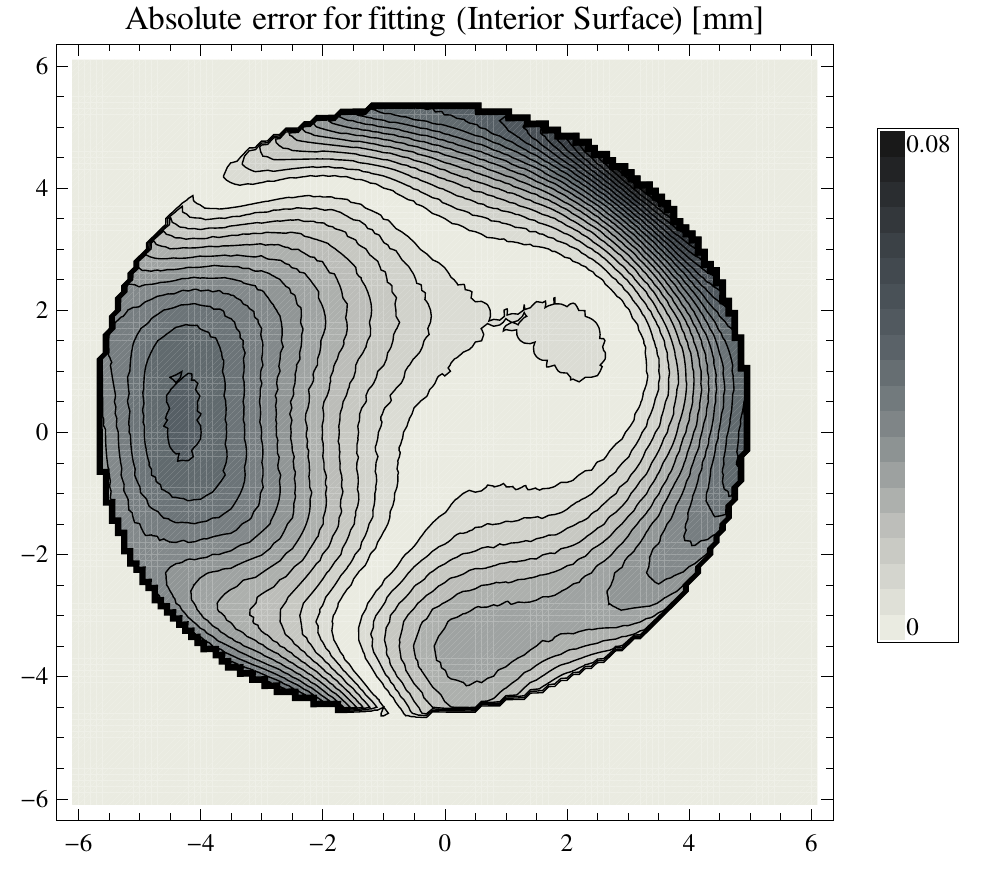}
	\caption{Absolute fitting errors for exterior (left) and interior (right) surfaces.}
	\label{Err}
\end{figure}

\section{Conclusion}
In this letter we have proposed a new mathematical model of corneal topography based on the modified fractional Bessel differential equation. We found its solution in a form of absolutely convergent power series. This solution is a generalization of the classical modified Bessel function of order 0. Also, we have found the behavior of $I_0^\alpha$ as $x\rightarrow\infty$, which reduces to the known cases as $\alpha=0,1$. When applying our model to the real corneal data we found accurate fit and parameter orders predicted by theory.

\section*{Acknowledgment}
The authors would like to thank Dr. Robert Iskander from Institute of Biomedical Engineering
and Instrumentation, Wroclaw University of Technology, Poland and School of Optometry, Queensland University of Technology, Australia for access to the data.

\end{document}